\documentclass[11pt]{amsart}
\usepackage[margin=1in]{geometry}
\usepackage{xcolor}
\usepackage[english]{babel}
\numberwithin{equation}{section}
\newtheorem{thm}{Theorem}[section]
\newtheorem{prop}[thm]{Proposition}
\newtheorem{lem}[thm]{Lemma}

\newtheorem{conj}[thm]{Conjecture}

\def\Xint#1{\mathchoice
	{\XXint\displaystyle\textstyle{#1}}%
	{\XXint\textstyle\scriptstyle{#1}}%
	{\XXint\scriptstyle\scriptscriptstyle{#1}}%
	{\XXint\scriptscriptstyle\scriptscriptstyle{#1}}%
	\!\int}
\def\XXint#1#2#3{{\setbox0=\hbox{$#1{#2#3}{\int}$ }
		\vcenter{\hbox{$#2#3$ }}\kern-.6\wd0}}

\def\dashint{\Xint-}
\newcommand{\rt}{\mathbb{R}^2}

\newcommand{\mnp}{(\m \cdot\nabla p) }

	\newcommand{\bv}{\mathbf{v}}
	\newcommand{\ot}{\Omega_T}
	
	\newcommand{\rn}{\mathbb{R}^N}

	\newcommand{\ob}{\partial\Omega}
	\newcommand{\pt}{\partial_t}
	
	\newcommand{\br}{B_{r}(x_0)}

	\newcommand{\bry}{B_{r}(y)}

	\newcommand{\ibr}{\int_{B_{r}(x_0)}}
	
	\newcommand{\ibry}{\int_{B_{r}(y)}}

	\newcommand{\qr}{Q_{r}(z_0)}

	\newcommand{\iqr}{\int_{Q_{r}(z_0)}}
	\newcommand{\iqrho}{\int_{Q_{\rho}(z_0)}}

	\newcommand{\ra}{\rightarrow}
	\newcommand{\io}{\int_{\Omega}}

	\newcommand{\m}{\mathbf{m}}
	\newcommand{\n}{\mathbf{n}}
	
	\newcommand{\ece}{(I+\m\otimes\m)}

\begin{document}
		\title[an elliptic-parabolic system modeling biological transportation network]{ H\"{o}lder continuity of weak solutions to an elliptic-parabolic system modeling biological transportation network}
		\author{Xiangsheng Xu}\thanks
		{Department of Mathematics and Statistics, Mississippi State
			University, Mississippi State, MS 39762.
			{\it Email}: xxu@math.msstate.edu.}
		\begin{abstract} In this paper we study the regularity of weak solutions to an elliptic-parabolic system modeling natural network formation. The system is singular and involves cubic nonlinearity. Our investigation reveals that weak solutions are H\"{o}lder continuous when the space dimension $N$ is $2$. This is achieved via an inequality associated with the Stummel-Kato class of functions and refinement of a lemma originally due to S. Campanato and C. B. Morrey (\cite{G}, p. 86).
		\end{abstract}
		\maketitle

		\section{Introduction}
		Let $\Omega$ be a bounded domain in $\rt$ with Lipschitz boundary $\ob$ and let $T$ be any positive number. 
		We study the behavior of weak solutions of the system
		\begin{eqnarray}
			-\mbox{{div}}\left[\ece\nabla p\right]&=& S(x)\ \ \ \mbox{in $\ot\equiv \Omega\times(0,T)$,}\label{e1}\\
			\partial_t\m -D^2\Delta \m +|\m |^{2(\gamma-1)}\m &=& E^2\mnp\nabla p\ \ \ \mbox{in $\ot$}\label{e2}
		\end{eqnarray}
		for given function $S(x)$ and physical parameters $D, E, \gamma$ with properties:
		\begin{enumerate}
			\item[(H1)] $S(x)\in L^q(\Omega), \ q>1$; and
			\item[(H2)] $D, E\in (0, \infty), \gamma\in (\frac{1}{2}, \infty)$.
		\end{enumerate}
		Here $\m =\m (x,t)=(m_1 (x,t), m_2 (x,t))^T$ is a vector-valued function and the out product
		$\m\otimes \m$ is the matrix whose $ij$-entry is $m_im_j$. That is, 
		$$\m\otimes \m=\m \m^T.$$
		Thus,
		\begin{equation}
			\m\otimes \m \nabla p=\mnp\m.\nonumber
		\end{equation}
		
		This system has been proposed by Hu and Cai \cite{H,HC} to describe 
		natural network formation. In this situation the scalar function $p=p(x,t)$ is the pressure due to Darcy's law, while $\m $ is the conductance vector. 
		The function $S(x)$ is the time-independent source term. Values of the parameters $D, E$, and $\gamma$ are
		determined by the particular physical applications one has in mind. For example, $\gamma =1$ corresponds to leaf venation \cite{H}. 
		Of particular physical interest is the initial boundary value problem: in addition to \eqref{e1} and \eqref{e2} one requires
		\begin{eqnarray}
			\m (x,0)&=&\m _0(x)\ \ \ \mbox{ on $\Omega$},\label{e3}\\
			p(x,t)=0, && \m (x,t)=0 \ \ \mbox{ on $\Sigma_T\equiv\partial\Omega\times(0,T)$}.\label{e4}
		\end{eqnarray}
	Here the initial data satisfies
		\begin{enumerate}
			\item[(H3)] $\m _0\in\left( W^{1,2}_0(\Omega)\cap L^{2\gamma}(\Omega)\right)^2$.
		\end{enumerate}
	%
			A pair $(\m , p)$ is said to be a weak solution to \eqref{e1}-\eqref{e4} if:
			\begin{enumerate}
				\item[(D1)] $\m \in L^\infty\left(0,T; \left(W^{1,2}_0(\Omega)\cap L^{2\gamma}(\Omega)\right)^2\right), \partial_t\m \in L^2(0,T; \left(L^2(\Omega)\right)^2), p\in L^\infty(0,T; W^{1,2}_0(\Omega)), \\  \m \cdot\nabla p \in L^\infty(0,T;  L^{2}(\Omega))$;
				\item[(D2)] $\m (x,0)=\m _0$ in $C\left([0,T]; \left(L^2(\Omega)\right)^2\right)$;
				\item[(D3)] Equations \eqref{e1} and \eqref{e2} are satisfied in the sense of distributions.
			\end{enumerate}
		Mathematical analysis of the problem has attracted tremendous attention recently, and we refer the reader to \cite{X6} and the references therein for more detailed information. Here we summarize the relevant known results in the following
		\begin{prop}\label{prop} Let (H1)-(H3) hold. Then there is a weak solution $(p,\m)$ to \eqref{e1}-\eqref{e4}.
			Moreover,
			\begin{equation}
				\underset{0\leq t\leq T}{\textup{ess sup}}\ \|	p\|_{\infty,\Omega}\leq c\|S\|_{q,\Omega} \ \ \mbox{for some $c>0$,}\label{hc4}\\
			\end{equation} 
			and for each $y\in\Omega$, $\ell>0$ there is a positive constant $c=c(\ell, d_y)$, where $d_y=\textup{dist}(y,\ob)$,  such that 
			\begin{eqnarray}
					\underset{0\leq t\leq T}{\textup{ess sup}}\	\underset{\bry}{\textup{osc}  }\ \ p&\equiv&	\underset{0\leq t\leq T} {\textup{ess sup}}\left(	\underset{\bry}{\textup{ess sup}  }\ \ p-	\underset{\bry}{\textup{ess inf}  }\ \ p\right)
				\leq \frac{c}{\ln^\ell\frac{R}{r}}, \ \ r\in (0,R],\label{hc2}\\
				\underset{0\leq t\leq T}{\textup{ess sup}}	\ibry F dx&\leq& \frac{c}{\ln^{2\ell-2}\frac{R}{r}}, \ \ r\in (0,R],\label{hc3}\\
				\underset{0\leq t\leq T}{\textup{ess sup}}\ \	\eta(F;\bry;r)&\equiv&\sup_{x_0\in\bry}\ibr F\chi_{\bry}|\ln|x-x_0||dx\nonumber\\
				&\leq& \frac{c}{\ln^{2\ell-3}\frac{R}{r}}, \ \ r\in (0,R],\label{kato2}
			\end{eqnarray}
		where  $\bry$ denotes the ball centered at $y$ with radius $r$, $R\in \left(0,\min\{1, \frac{1}{2}d_y\}\right)$, and
		$$F=|\nabla p|^2+\mnp^2.$$
		\end{prop}
	The existence part of the proposition was established by Haskovec, Markowich, and Perthame \cite{HMP} for any space dimensions, while \eqref{hc4} was contained in \cite{LX}. Recently, the author obtained \eqref{hc2}-\eqref{kato2} in \cite{X8}. What is so remarkable is that $\ell$ in \eqref{hc2} can be arbitrarily large. This plays an essential role in our later development.
We easily see from the proof in \cite{X8} that \eqref{hc3} and \eqref{kato2} are actually consequences of \eqref{hc2}. 
		However, high regularity for $\m$ is largely open. In fact, Theorem 3.1 in \cite{X6} says that we need to strengthen
	\eqref{hc2} to 
	\begin{equation}\label{main}
		p\in L^\infty(0, T; C^\alpha_{\textup{loc}}(\Omega)) \ \ \mbox{for some $\alpha$ in $(0,1)$}
	\end{equation} in order to gain
\begin{equation}
	\m\in \left(C^{\beta, \frac{\beta}{2}}_{\textup{loc}}(\ot)\right)^2\ \ \mbox{for some $\beta$ in $(0,1)$.}\label{main2}
\end{equation}
	In this paper we shall  bridge the gap between \eqref{hc2} and \eqref{main} for the problem. To be precise, we have:
		\begin{thm}\label{mth}
			Let (H1)-(H3) be satisfied and $(p,\m)$ be a weak solution to \eqref{e1}-\eqref{e4}.
			Then \eqref{hc2} implies \eqref{main}. As a result, \eqref{main2} also holds.
		\end{thm}
		
		In \cite{X4}, the author claimed to have established the global existence of a strong solution for the case $N=2$. Of course, a strong solution satisfies \eqref{main} and \eqref{main2}. However, since inequality (73) in \cite{X4} was used, the result actually required that the initial data be suitably small. Therefore, it was a small data global existence.

		The mathematical difficulty of our problems is due to the terms $\m\otimes\m\nabla p$ in \eqref{e1} and $\mnp\nabla p$ in \eqref{e2}. They represent the so-called ``cubic nonlinearity''\cite{LX}. 
		In fact, our problem here is related to a conjecture by De Giorgi \cite{C,DE}. In 1995, De Giorgi gave a lecture in Lecce, Italy on the continuity of weak solutions to second-order elliptic equations of the form
		\begin{equation}\label{ell1}
			\mbox{div}(A\nabla u)=0\ \ \mbox{in $\Omega$}.
		\end{equation}
	The coefficient matrix $A=A(x)$ satisfies
	\begin{equation}
		\lambda(x)|\xi|^2\leq A(x)\xi\cdot\xi\leq \Lambda(x)|\xi|^2, \ \xi\in\rn, x\in\Omega,\nonumber
	\end{equation}
for some non-negative functions $\lambda(x), \Lambda(x)$. If we can take $\lambda(x)$ to be a constant and $\Lambda(x)$ is not bounded above, we say that \eqref{ell1} is singular. According to this definition,  our equation \eqref{e1} is singular because
\begin{equation}
|\xi|^2\leq	(I+\m\otimes\m)\xi\cdot\xi=|\xi|^2+(\m\cdot\xi)^2\leq (1+|\m|^2)|\xi|^2\ \ \mbox{for a.e $x\in \Omega$ and each $\xi\in \rn$.}\nonumber
\end{equation}
and (D1) does not imply the boundedness of $\m$.
		One of De Giorgi's open problems for singular equations was stated as follows: 
		\begin{conj}
			If $N\geq 3$, $\lambda=1$, and 
			\begin{equation}\label{ell2}
				\io e^{\Lambda(x)}dx<\infty,\nonumber
			\end{equation}
				then weak solutions of \eqref{ell1} are continuous.
		\end{conj}
		This conjecture remains open. In our case,
		 (D1) combined with $N=2$ and Theorem 7.15 in \cite{GT} asserts that 
		$$\io e^{c_0|\m|^2}dx<\infty\ \ \mbox{for some positive number $c_0$}.$$
		However, the case $N=2$ is not included in the above conjecture.  On the other hand, the bare continuity of $u$ is not enough for our purpose.
	
		
		If $\Lambda\leq c\lambda$ for some $c>0$ and $\lambda$ is an $A_2$ weight then weak solutions of \eqref{ell1} are continuous no matter what the space dimension $N$ is \cite{HKM}.
	If $\Lambda$ is a constant and $\lambda$ is not bounded away from $0$ below,	then \eqref{ell1} is called degenerate. We refer the reader to \cite{OZ} for a continuity result in this case. 
	
	Our approach is based upon two elements. The first is an inequality associated with $ K_2(\Omega)$, the Stummel-Kato class of functions \cite{CFG,K}. 
	We say that $f\in  K_2(\Omega)$  if  $f$ is a measurable function on  $\Omega$ and
	\begin{equation}\label{kato1}
		\eta(f;\Omega;r)\equiv\sup_{y\in \Omega}\ibry|f(x)|\chi_{\Omega}\left|\ln|x-y|\right| dx\ra0\ \ \mbox{ as $r\ra 0^+$},
	\end{equation}
	while
	$f\in K_2^{\textup{loc}}(\Omega)$ means that $\lim_{r\ra 0}\eta(f;\Omega_1;r)=0$ for each bounded subdomain $\Omega_1$ of $\Omega$ with $\overline{\Omega_1}\subset\Omega$. Note that in \eqref{kato1} we have used $y\in\Omega$ instead of $y\in \rt$ as was done in \cite{K}. Our definition here seems to be more suitable for PDE applications.
	As usual, the letter $c$ or $c_i, i=0,1,\cdots$, will be used to represent a generic positive constant.
	\begin{lem}\label{kato3}There is a constant $c$ such that
		\begin{equation}
			\ibr|f|v^2dx\leq c\eta(f;\br;r)\ibr|\nabla v|^2dx\ \ \mbox{for each $v\in W^{1,2}_0(\br)$.}\nonumber
		\end{equation}
	\end{lem}

	 This lemma is simpler than Lemma 1.1 in \cite{FGL}, and its proof can be found in \cite{X8}. When it is applied to the function $F$ in Proposition \ref{prop} with \eqref{kato2} in mind, we yield important consequences for our analysis.
		
			We would like to point out that there is a parabolic version \cite{Z} of the elliptic theory in \cite{CFG,K}. At first glance it seems to be natural for us to try to apply the results in \cite{Z} to our problem here. This effort has failed due to the fact that \eqref{e1} is degenerate in $t$.
			
	The second ingredient in our approach is the following
		\begin{thm}\label{kl}
			Assume that  $\phi(t)$ is a non-negative and nondecreasing function on $(0,R_0]$ for some $R_0>0$, satisfying
			\begin{equation}\label{kl1}
				\phi(\rho)\leq A\left(\frac{\rho}{R}\right)^\alpha\phi(R)+\frac{BR^\beta}{\ln^\ell\frac{R_0}{R}} \ \ \mbox{ for all $0<\rho\leq R\leq R_0$,}
			\end{equation}
			where  $A, B,\alpha, \beta$, and $\ell$ are non-negative constants with $\alpha>\beta$. Then there exist $\gamma_0\in (\beta,\alpha),\tau\in(0,1)$, and $c>0$, all of which are determined by the given constants in \eqref{kl1}, such that
			\begin{equation}\label{kl5}
				\phi(\rho)\leq c\left(\frac{\rho}{R}\right)^{\gamma_0}\left(\phi( R)+\frac{R^\beta}{\ln^\ell\frac{R_0}{R}}\right)+\frac{c\rho^\beta}{\ln^\ell\frac{\tau^2R_0}{\rho}}\ \ \mbox{ for all $0<\rho\leq R\leq \tau^2R_0$}.
			\end{equation}
		\end{thm}	
	There are various versions of this theorem, the earliest of which can be found in \cite{MO}. Our version here
may be viewed as a refinement of Lemma 2.1 in (\cite{G}, p.86). Results of this kind have played important roles in the study of regularity properties of partial differential equations. Here it enables us to show that
for each compact subset $\mathcal{K}$ of $\ot$ and $\ell>0$ there exist $c>0$ and $ \rho_0>0$
such that
	\begin{equation}\label{vkll}
		\dashint_{Q_\rho(z_0)}|\m-\m_{z_0,\rho}|^2dxdt\leq \frac{c}{\ln^\ell\frac{1}{\rho}}\ \ \mbox{for all $\rho\in(0,\rho_0)$ and $z_0=(x_0,t_0)\in \mathcal{K}$,}
	\end{equation}
	where $Q_\rho(z_0)=B_\rho(x_0)\times\left(t_0-\frac{1}{2}\rho, t_0+\frac{1}{2}\rho\right)$ and
	$$\m_{z_0,\rho}=\dashint_{Q_\rho(z_0)}\m\ dxdt=\frac{1}{|Q_\rho(z_0)|}\iqr\m\ dxdt.$$
This is enough for the local boundedness of $\m$ (see Lemma \ref{logb} below). 

		As indicated in (\cite{R}, p. 82), \eqref{main2} is sufficient for 
		$$|\nabla p|\in L^\infty(0,T; L^s_{\mbox{loc}}(\Omega))\ \ \ \mbox{for each $s\geq 2$ whenever $S\in L^{\frac{Ns}{N+s}}_{\mbox{loc}}(\Omega) $}.$$
		This, in turn, results in higher regularity for $\nabla\m$. We shall not pursue the details here.
		We do not believe that it is very difficult to extend our estimates to the boundary. We shall leave this to the interested reader.

		
		Several partial regularity theorems were obtained in \cite{LX,X6}. According to Theorem \ref{mth}, the singular sets in \cite{LX,X6} are empty when $N=2$. The life-span of a strong solution in high space dimensions was investigated in \cite{X1}. See \cite{L,X5} and the references therein for other related studies. 	Note that the question of existence in the case where $\gamma=\frac{1}{2}$ is addressed in \cite{HMPS}. This is when the term $|\m |^{2(\gamma-1)}\m$ loses its continuity at $\m =0$. 
		It must be replaced by the following multi-valued function 
		$$g(x,t)=\left\{\begin{array}{ll}
			|\m |^{2(\gamma-1)}\m  & \mbox{if $\m \ne 0$,}\\
			\in [-1,1]^2 &\mbox{if $\m \ne 0$.}
		\end{array}\right.$$

The rest of the paper will be devoted to the proof of Theorems \ref{mth} and \ref{kl}.

	\section{Proof of main results}
	We shall begin the section with the proof of Theorem \ref{kl} and end it with the proof of Theorem \ref{mth}. Before we start the proof, recall that for any $\alpha>0, \ell>0, R_0>0$ there is a constant $c$ such that
	$$\sup_{0\leq r\leq R_0}r^\alpha\ln^\ell\frac{R_0}{r}\leq c.$$
	We will use this without acknowledgment.

	\begin{proof}[Proof of Theorem \ref{kl}] Without loss of any generality, assume that $A>1$. Then we  pick $\tau\in(0,1) $ so that
		$$	\gamma_0\equiv\alpha+\frac{\ln A}{\ln\tau}>\beta.$$
		Let $\tau$ be so chosen. For any $R\in (0, R_0)$ take $\rho=\tau R$ in \eqref{kl1} to get
		\begin{eqnarray}
			\phi(\tau R)&\leq& A\tau^\alpha\phi( R)+\frac{BR^\beta}{\ln^\ell\frac{R_0}{R}}\nonumber\\
			&=&\tau^{\gamma_0}\phi( R)+\frac{BR^\beta}{\ln^\ell\frac{R_0}{R}}.\nonumber
		\end{eqnarray}
		For any non-negative integer $k$ we replace $R$ by $\tau^k R$ in the above inequality to derive
		\begin{eqnarray}
			\phi(\tau^{k+1} R)&\leq& \tau^{\gamma_0}\phi( \tau^kR)+\frac{BR^\beta\tau^{\beta k}}{\ln^\ell\frac{R_0}{\tau^kR}}.\nonumber
		\end{eqnarray}
		Iterate over $k$ to get
		\begin{eqnarray}
			\phi(\tau^{k+1} R)&\leq &\tau^{(k+1)\gamma_0}\phi( R)+BR^\beta\tau^{k\beta}\sum_{j=0}^{k}\frac{\tau^{j(\gamma_0-\beta)}}{\ln^\ell\frac{R_0}{\tau^{k-j}R}}.\label{kl2}
		\end{eqnarray}
		We easily check that 
		\begin{eqnarray}
			\sum_{j=0}^{k}\frac{\tau^{j(\gamma_0-\beta)}}{\ln^\ell\frac{R_0}{\tau^{k-j}R}}&=&\frac{1}{\ln^\ell\frac{R_0}{\tau^kR}}\sum_{j=0}^{k-1}\frac{\tau^{j(\gamma_0-\beta)}\ln^\ell\frac{R_0}{\tau^kR}}{\ln^\ell\frac{R_0}{\tau^{k-j}R}}+\frac{\tau^{k(\gamma_0-\beta)}}{\ln^\ell\frac{R_0}{R}}.\label{kl3}
		\end{eqnarray}
		For $ j=0,1,\cdots, k-1$, we have
		\begin{eqnarray}
			\frac{\ln\frac{R_0}{\tau^kR}}{\ln\frac{R_0}{\tau^{k-j}R}}&=&\frac{\ln R_0-k\ln\tau-\ln R}{\ln R_0-(k-j)\ln\tau-\ln R}
			\nonumber\\
			&=&1+\frac{-j\ln\tau}{-(k-j)\ln\tau+\ln \frac{R_0}{R}}\nonumber\\
			&\leq& 1+\frac{j}{k-j}\leq 1+j\ \ \mbox{ due to $ \ln \frac{R_0}{R}>0$, and}\label{kl4}\\
			\sum_{j=0}^{\infty}(1+j)^\ell\tau^{j(\gamma_0-\beta)}&<&\infty\label{kl6}.
		\end{eqnarray}
		Use \eqref{kl3}, \eqref{kl4}, and \eqref{kl6} in \eqref{kl2} to deduce
		\begin{eqnarray}
			\phi(\tau^{k+1} R)&\leq&\tau^{(k+1)\gamma_0}\phi( R) +BR^\beta\tau^{k\beta}\left(\frac{1}{\ln^\ell\frac{R_0}{\tau^kR}}\sum_{j=0}^{k-1}(1+j)^\ell\tau^{j(\gamma_0-\beta)}+\frac{\tau^{k(\gamma_0-\beta)}}{\ln^\ell\frac{R_0}{R}}\right)\nonumber\\
			&\leq&\tau^{(k+1)\gamma_0}\left(\phi( R)+\frac{BR^\beta}{\tau^{\gamma_0}\ln^\ell\frac{R_0}{R}}\right)+\frac{cBR^\beta\tau^{k\beta}}{\ln^\ell\frac{R_0}{\tau^kR}}.\nonumber
		\end{eqnarray}
		Now let $R\in(0,\tau^2R_0)$. 
		For each $\rho\in (0,R]$ there is a non-negative integer $k$ such that
		$$\tau^{k+1} R<\rho\leq\tau^kR.$$
		If $k\geq 1$ then
		\begin{eqnarray}
			\phi(\rho)&\leq& \phi(\tau^kR)\nonumber\\
			&\leq& \tau^{k\gamma_0}\left(\phi( R)+\frac{BR^\beta}{\tau^{\gamma_0}\ln^\ell\frac{R_0}{R}}\right)+\frac{cBR^\beta\tau^{(k-1)\beta}}{\ln^\ell\frac{R_0}{\tau^{k-1}R}}\nonumber\\
			&\leq&\frac{1}{\tau^{\gamma_0}}\left(\frac{\rho}{R}\right)^{\gamma_0}\left(\phi( R)+\frac{BR^\beta}{\tau^{\gamma_0}\ln^\ell\frac{R_0}{R}}\right)+\frac{cBR^\beta}{\tau^{2\beta}\ln^\ell\frac{\tau^2R_0}{\rho}}\left(\frac{\rho}{R}\right)^\beta.\nonumber
		\end{eqnarray}
		If $k=0$, the above inequality is obviously true.
		This yields \eqref{kl5}.
	\end{proof}
	
	The proof of Theorem \ref{mth} is divided into two lemmas. Before we begin, we introduce some notations. The parabolic distance between $z_1=(x_1,t_1)$ and $z_2=(x_2,t_2)$, represented by $\mbox{dist}_p(z_1, z_2)$, is defined to be
	$$\mbox{dist}_p(z_1, z_2)=|x_1-x_2|+\sqrt{|t_1-t_2|}.$$
	The parabolic boundary of $\ot$, denoted by $\partial_p\ot$, is the set $\Sigma_T\cup \Omega\times \{0\}$.
\begin{lem}\label{vmo} For each $z_0=(x_0, t_0)\in\ot$ and $\ell>0$ there exist $c>0$ and $ \rho_0>0$ determined by $\textup{dist}_p(z_0, \partial_p\ot)$ and other given data such that
	\begin{equation}\label{vkl}
		\dashint_{Q_\rho(z_0)}|\m-\m_{z_0,\rho}|^2dxdt\leq \frac{c}{\ln^\ell\frac{1}{\rho}}\ \ \mbox{for all $\rho\in(0,\rho_0)$.}
	\end{equation}
	\end{lem}	
\begin{proof}
	Let $z_0=(x_0, t_0)\in\ot$ be given. For each $r>0$ such that $\qr\subset\ot$ we consider the initial boundary value problem
	\begin{eqnarray}
		\pt\n-D^2\Delta\n&=& 0\ \ \mbox{in $\qr$},\label{ne1}\nonumber\\
		\n&=&\m\ \ \mbox{on $\partial_p\qr$}.\nonumber
	\end{eqnarray}
	According to Claim 1 in \cite{X7}, there exist $c>0, \alpha\in (0,1)$ such that
	\begin{equation}\label{kl10}
		\iqrho|\n-\n_{z_0,\rho}|^2dxdt\leq c\left(\frac{\rho}{r}\right)^{4+2\alpha}	\iqr|\n-\n_{z_0,r}|^2dxdt\ \ \mbox{for $0<\rho\leq r$.}
	\end{equation}
	
Let 
\begin{equation}
	\bv=\m-\n.\nonumber
\end{equation}
Then $\bv$ satisfies the problem
\begin{eqnarray}
	\pt\bv-D^2\Delta\bv&=&-|\m|^{2(\gamma-1)}\m+E^2\mnp\nabla p\ \ \mbox{in $\qr$},\label{ve1}\\
	\bv&=&\mathbf{0}\ \ \mbox{on $\partial_p\qr$}.\label{ve2}
\end{eqnarray}
Use $\bv$ as a test function in \eqref{ve1} to derive
\begin{eqnarray}
	\lefteqn{\frac{1}{2}\frac{d}{dt}\ibr|\bv|^2dx+D^2\ibr|\nabla\bv|^2dx}\nonumber\\
	&=&-\ibr|\m|^{2(\gamma-1)}\m\cdot\bv dx+ E^2\ibr \mnp\nabla p\cdot\bv dx\nonumber\\
	&\equiv& I_1+I_2.\label{ho3}
\end{eqnarray}
For each $s\geq 2$ we derive from Poincar\'{e}'s inequality that
\begin{eqnarray}
	I_1&\leq&\||\m|^{2\gamma-1}\|_{\frac{s}{s-1},\br}\left(\ibr|\bv|^sdx\right)^{\frac{1}{s}}\nonumber\\
	&\leq&c\||\m|^{2\gamma-1}\|_{\frac{s}{s-1},\br}\left(\ibr|\nabla\bv|^{\frac{2s}{2+s}}dx\right)^{\frac{2+s}{2s}}\nonumber\\
	&\leq&cr^{\frac{2}{s}}\||\m|^{2\gamma-1}\|_{\frac{s}{s-1},\br}\left(\ibr|\nabla\bv|^{2}dx\right)^{\frac{1}{2}}\nonumber\\
		&\leq&\frac{1}{4}\ibr|\nabla\bv|^{2}dx+cr^{\frac{4}{s}}\||\m|^{2\gamma-1}\|_{\frac{s}{s-1},\br}^2.\label{ho1}
\end{eqnarray}
It follows from \eqref{ve2} and Lemma \ref{kato3} that
\begin{equation}
	\ibr F|\bv|^2dx\leq c\eta(F;\br; r)\ibr|\nabla\bv|^2dx,\nonumber
\end{equation}
where $F$ is given as in Proposition \ref{prop}. Apply the proposition appropriately to obtain that for each $\ell>0$ and some suitably small $R>0$ there is a positive  $c$ with
\begin{equation}
		\ibr F|\bv|^2dx\leq \frac{c}{\ln^\ell\frac{R}{r}}\ibr|\nabla\bv|^2dx,\ \ r\in (0, R).\nonumber
\end{equation}
With this and \eqref{hc3} in mind, we estimate that
\begin{eqnarray}
	I_2&\leq&\frac{E^2}{2}\ibr F|\bv|dx\nonumber\\
	&\leq&\frac{E^2}{2}\int_{\br\cap\{|\bv\leq 1\}} F|\bv|dx+\frac{E^2}{2}\int_{\br\cap\{|\bv> 1\}} F|\bv|dx\nonumber\\
	&\leq&\frac{E^2}{2}\ibr Fdx+\frac{E^2}{2}\ibr F|\bv|^2dx\nonumber\\
	&\leq&\frac{c}{\ln^\ell\frac{R}{r}}+\frac{c}{\ln^\ell\frac{R}{r}}\ibr|\nabla\bv|^2dx.\label{ho2}
\end{eqnarray}
Collect \eqref{ho2} and \eqref{ho1} in \eqref{ho3} and choose $R$ suitably small in the resulting inequality to deduce
\begin{eqnarray}
	\sup_{t_0-\frac{1}{2}r^2\leq t\leq t_0+\frac{1}{2}r^2}\ibr|\bv|^2dx+\iqr|\nabla\bv|^2dxdt\leq cr^{2+\frac{4}{s}}+\frac{cr^2}{\ln^\ell\frac{R}{r}},\ \ r\in(0,R).\nonumber
\end{eqnarray} 
Let $s$ be given as before. We derive from Poincar\'{e}'s inequality that
\begin{eqnarray}
	\iqr|\bv|^2dxdt
	&\leq&\int_{t_0-\frac{1}{2}r^2}^{t_0+\frac{1}{2}r^2}\left(\ibr |\bv|^{\frac{s}{s-1}}dx\right)^{\frac{s-1}{s}}\left(\ibr|\bv|^sdx\right)^{\frac{1}{s}}dt\nonumber\\
	&\leq&cr^{\frac{s-2}{s}}\left(\sup_{t_0-\frac{1}{2}r^2\leq t\leq t_0+\frac{1}{2}r^2}\ibr|\bv|^2dx\right)^{\frac{1}{2}}\int_{t_0-\frac{1}{2}r^2}^{t_0+\frac{1}{2}r^2}\left(\ibr|\nabla\bv|^{\frac{2s}{s+2}}\right)^{\frac{s+2}{2s}}dt\nonumber\\
	&\leq&cr^2\left(\sup_{t_0-\frac{1}{2}r^2\leq t\leq t_0+\frac{1}{2}r^2}\ibr|\bv|^2dx\right)^{\frac{1}{2}}\left(\iqr|\nabla\bv|^{2}dxdt\right)^{\frac{1}{2}}\nonumber\\
		&\leq& cr^{4+\frac{4}{s}}+\frac{cr^4}{\ln^\ell\frac{R}{r}}\leq \frac{cr^4}{\ln^\ell\frac{R}{r}}.\label{kl11}
\end{eqnarray}
Set
$$\phi(\rho)=\iqrho|\m-\m_{z_0,\rho}|^2dxdt.$$
We easily check that $\phi(\rho)$ is an increasing function of $\rho$. Moreover, for $\rho\leq r$ we deduce from \eqref{kl10} and \eqref{kl11} that
\begin{eqnarray}
	\phi(\rho)&\leq&2\iqrho|\n-\n_{z_0,\rho}|^2dxdt+2\iqrho|\bv-\bv_{z_0,\rho}|^2dxdt\nonumber\\
	&\leq&c\left(\frac{\rho}{r}\right)^{4+2\alpha}	\iqr|\n-\n_{z_0,r}|^2dxdt+c\iqr|\bv|^2dxdt\nonumber\\
		&\leq&c\left(\frac{\rho}{r}\right)^{4+2\alpha}	\iqr|\m-\m_{z_0,r}|^2dxdt+c\iqr|\bv|^2dxdt\nonumber\\
		&\leq&c\left(\frac{\rho}{r}\right)^{4+2\alpha}	\phi(r)+\frac{cr^4}{\ln^\ell\frac{R}{r}}.\nonumber
\end{eqnarray}
Now we are in a position to invoke Theorem \ref{kl}. Upon doing so, we obtain $\gamma_0\in(4, 4+2\alpha)$, $\tau\in (0,1)$, $R_0>0$  such that
\begin{equation}
		\phi(\rho)\leq c\left(\frac{\rho}{R}\right)^{\gamma_0}\left(\phi( R)+\frac{R^4}{\ln^\ell\frac{R_0}{R}}\right)+\frac{c\rho^4}{\ln^\ell\frac{\tau^2R_0}{\rho}}\ \ \mbox{ for all $0<\rho\leq R\leq \tau^2R_0$}.\nonumber
\end{equation}
This implies \eqref{vkll} because $\gamma_0>4$ and the constant in the right-hand side of the above inequality depends only on $\textup{dist}_p(z_0,\partial_p\ot)$, not $z_0$ itself. The proof is complete.
\end{proof}
\begin{lem}\label{logb} We have
	$$|\m|\in L^\infty_{\textup{loc}}(\ot).$$
\end{lem}
\begin{proof}We follow the proof of Theorem 1.2 in (\cite{G}, p. 70). Let $z_0,\rho_0$ be given as in Lemma \ref{vmo}. Set
	$$R_i=\frac{\rho_0}{2^i},\ \  i=0, 1, \cdots.$$ We claim that $\{\m_{z_0,R_i}\}$ is a Cauchy sequence. To see this, we
	integrate the inequality  
	$$|\m_{z_0,R_i}-\m_{z_0,R_{i+1}} |\leq|\m_{z_0,R_i}-\m(x,t)|+|\m(x,t)-\m_{z_0,R_{i+1}}|.$$
	over $Q_{R_{i+1}}(z_0)$ to derive
$$|\m_{z_0,R_i}-\m_{z_0,R_{i+1}} |\leq	 c\dashint_{Q_{R_i}(z_0)}|\m_{z_0,R_i}-\m|dx+\dashint_{Q_{R_{i+1}}(z_0)}|\m-\m_{z_0,R_{i+1}}|dx.$$
Apply \eqref{vkl} with $\ell>1$ and $\rho_0<1$ to deduce
	\begin{equation}\label{logb1}
		|\m_{z_0,R_i}-\m_{z_0,R_{i+1}} |\leq\frac{c}{i^\ell},
			\end{equation}
from whence the claim follows. We can define
$$\tilde{\m}(z_0)=\lim_{i\ra\infty}\m_{z_0,R_i}.$$
We must show
\begin{equation}
	\lim_{\rho\ra 0}\m_{z_0,\rho}=\tilde{\m}(z_0).\nonumber
\end{equation}
To this end, we observe that
	for each $\rho\in (0, \rho_0)$ there is an $i\in \{0,1,\cdots\}$ such that
	\begin{equation}\label{j42}
		\frac{\rho_0}{2^{i+1}}\leq \rho<\frac{\rho_0}{2^i}.
	\end{equation}
	By a calculation similar to \eqref{logb1}, we have
	\begin{eqnarray}
		|\m_{z_0,\rho}-\m_{z_0,R_i}|&\leq &
		c\dashint_{Q_{R_i}(z_0)}|\m_{z_0,R_i}-\m|dx+\dashint_{Q_{\rho}(z_0)}|\m-\m_{z_0,\rho}|dx\leq \frac{c}{i^\ell}
	.\nonumber
	\end{eqnarray}
Note from \eqref{j42} that $i\ra \infty$ as $\rho\ra 0$. Subsequently, 
$$|\tilde{\m}(z_0)-\m_{z_0,\rho}|\leq |\tilde{\m}(z_0)-\m_{z_0,R_i}|+|\m_{z_0,R_i}-\m_{z_0,\rho}|\ra 0\ \ \mbox{as $\rho\ra 0$}.$$
Thus,
$$\tilde{\m}(z_0)=\m(z_0)\ \ \mbox{a.e. in $\ot$.}$$ Finally, we conclude from \eqref{logb1} that
\begin{eqnarray}
	|\m_{z_0,R_{i+1}}-\m_{z_0,\rho_0}|&\leq&\sum_{j=0}^{i}|\m_{z_0,R_j}-\m_{z_0,R_{j+1}} |\nonumber\\
	&\leq &c\sum_{j=1}^{i}\frac{1}{j^\ell}+\frac{c}{\ln^\ell\frac{1}{\rho_0}}\leq c+\frac{c}{\ln^\ell\frac{1}{\rho_0}}.\nonumber
\end{eqnarray}
Take $i\ra \infty$ in the above inequality to obtain
$$|\m (z_0)|\leq |\m_{z_0,\rho_0}|+ c+\frac{c}{\ln^\ell\frac{1}{\rho_0}}.$$
Once again, since the constant in the right-hand side of the above inequality depends only on $\textup{dist}_p(z_0,\partial_p\ot)$, not $z_0$ itself, the lemma follows.
\end{proof}
\begin{proof}[Conclusion of the proof of Theorem \ref{mth}]
		To complete the proof of Theorem \ref{mth}, we easily conclude from the classical elliptic regularity theory that the preceding lemma implies \eqref{main}. This together with Theorem 3.1 in \cite{X6} yields \eqref{main2}. This finishes the proof.\end{proof}

\end{document}